\documentclass[11pt,oneside]{article}
\usepackage[utf8]{inputenc}
\usepackage[english]{babel}
\usepackage{amsmath}
\usepackage{amsfonts}
\usepackage{amssymb}
\usepackage{amsthm}
\usepackage{tikz-cd} 
\usepackage[left=3cm,right=3cm,top=3cm,bottom=3cm]{geometry}
\usepackage[title]{appendix}
\usepackage[]{algorithm2e}

\usepackage[normalem]{ulem}
\useunder{\uline}{\ul}{}
\theoremstyle{plain}
\newtheorem{teo}{}[section]
\newtheorem{prop}[teo]{Proposition}

\newtheorem{rem}[teo]{Remark}
\newtheorem{lem}[teo]{Lemma}
\newtheorem{thm}[teo]{Theorem}
\theoremstyle{definition}
\newtheorem{ex}[teo]{Example}

\newcommand\blfootnote[1]{%
  \begingroup
  \renewcommand\thefootnote{}\footnote{#1}%
  \addtocounter{footnote}{-1}%
  \endgroup
}

\title{Shape of compacta as extension of weak homotopy of finite spaces}
\author{Pedro J. Chocano, Manuel A. Morón and Francisco R. Ruiz del Portal}
\date{}

\begin{document}
\maketitle

\begin{abstract}
We construct a category that classifies compact Hausdorff spaces by their shape and finite topological spaces by their weak homotopy type. 
\end{abstract}

\section{Introduction}\label{sec:introduccion}
\blfootnote{2020  Mathematics  Subject  Classification: 55P55,  	54C56, 06A06, 54C56, 18Fxx 	 }
\blfootnote{Keywords: finite topological spaces, poset, compact Hausdorff spaces, shape theory, weak homotopy type.}
\blfootnote{This research is partially supported by Grants PGC2018-098321-B-100 and BES-2016-076669 from Ministerio de Ciencia, Innovación y Universidades (Spain).}

Recent research recognizes the critical role played by the theory of finite topological spaces in several fields of mathematics such as dynamical systems \cite{barmak2011lefschetz,lipinski2019conley, chocano2020coincidence}, group theory \cite{barmak2009automorphism,barmak2020automorphism2,chocano2021onSome} algebraic topology (see \cite{barmak2011algebraic,may1966finite} and the references given there) and geometric topology \cite{mondejar2020reconstruction,chocano2020computational}. It is worth pointing out that important conjectures can be stated in terms of the theory of finite topological spaces. For instance, the Quillen's conjecture (see \cite[Chapter 8]{barmak2011algebraic} for more information). 

Recently, results about the reconstruction of compact metric spaces and its algebraic invariants point out that there are connections between shape theory and the theory of finite topological spaces (see \cite{moron2008connectedness,mondejar2020reconstruction,chocano2020computational}).  Previously, some related results connecting shape theory, and some of its invariants, with finiteness can be found in \cite{sanjurjo1997density,moron2001finite}. Shape theory provides a weaker classification than the one given by the classical homotopy theory of topological spaces. There are several approaches to shape theory, see for example \cite{borsuk1971shape,mardevsic1982shape,cordier1989shape,sanjurjo1992intrinsic}. The intuitive idea about this theory is to enlarge the set of morphisms of the homotopy category of topological spaces $HTop$. To this end, for every topological space it is associated an inverse system of polyhedra. Then the morphisms are given in terms of these inverse systems. The shape category $Sh$ gives the same classification than the homotopy category of polyhedra $HPol$, that is, there is a faithful functor $S:HTop\rightarrow Sh$ that induces an isomorphism between $HPol$ and the full subcategory of $Sh$ restricted to objects of $HPol$ (see \cite[Chapter 1, Section 3]{mardevsic1982shape}). Roughly speaking, we can say that shape theory deals with the study or classification of topological spaces with ``bad'' local properties, while it preserves the homotopy theory for topological spaces with ``good'' local properties. Therefore, one of the most important applications of this theory is to the study of dynamical systems, for a recent account of this we refer the reader to \cite{gabites2008dynamical}.

However, shape theory does not work properly to classify finite topological spaces. Every connected finite topological space is trivial in the shape category. In fact, the shape classification is characterized by the number of connected components of a finite topological space. Moreover, the notion of homotopy can be very restrictive for finite spaces. Indeed, for an important family of finite topological spaces (minimal finite spaces) we have that a homotopy equivalence is the same as a homeomorphism (see \cite{stong1966finite}). In addition, if $X$ is a finite topological space, then the finite barycentric subdivision of $X$, which is an analog of the barycentric subdivision for simplicial complexes, does not have the same homotopy type of $X$. Besides that, finite spaces have the same weak homotopy type of simplicial complexes (see \cite{mccord1966singular}) and hence this notion seems more reasonable to classify them.


The aim of this short note is to construct, in a simple manner, a category that classifies compact Hausdorff spaces by their shape and finite topological spaces by their weak homotopy type. We think that this paper also clarifies the results obtained in \cite{prasolov2010quasi}, in the more geometrical class of compact hausdorff spaces. The organization of this note is as follows. In Section \ref{sec:preliminares} we set up notation and review some of the standard facts on finite topological spaces. As for prerequisites, the reader is expected to be familiar with shape theory and the notion of pro-category. The best general reference here is \cite{mardevsic1982shape}. Section \ref{sec:construccionCategoria} is devoted to the construction of the desired category.

\section{Preliminaries}\label{sec:preliminares}

Let $F$ denote the category of finite topological $T_0$-spaces, its objects are all finite topological $T_0$-spaces and its morphisms are continuous maps. From now on we make the assumption: every finite topological space satisfies the $T_0$ separation axiom. The assumption of being a $T_0$-space is crucial. If $X$ is a finite topological $T_1$-space, then $X$ has the discrete topology.  Let $Poset$ denote the category of finite partially ordered sets (posets), its objects are finite posets and its morphisms are order-preserving maps. From the results of \cite{alexandroff1937diskrete}, it can be deduced the following theorem.
\begin{thm}\label{thm:alexandroffTheorem}
$F$ and $Poset$ are isomorphic categories.
\end{thm}

Given a finite topological space $X$ and $x\in X$, $U_x$ denotes the intersection of every open set containing $x$. It is clear that $U_x$ is an open set. We say that $x\leq y$ if and only if $U_x\subseteq U_y$. This defines a partial order on $X$. Hence, $U_x=\{ y\in X|y\leq x\}$. The notion of homotopy can be codified in a combinatorial way.
\begin{prop}\label{prop:characterizationHomotopy} Let $f,g:X\rightarrow Y$ be continuous maps between finite topological spaces. Then $f$ is homotopic to $g$ if and only if there exists a sequence of continuous maps $\{f_i:X\rightarrow Y\}_{i=1,...,n}$ such that $f_0=f$, $f_n=g$ and $f_0(x)\leq f_1(x)\geq ...\leq f_{n-1}(x)\geq f_n(x)$ for every $x\in X$.
\end{prop}

A useful way to visualize finite topological spaces is with Hasse diagrams. Let $X$ be a finite topological space (poset). The Hasse diagram of $X$ is a directed graph whose vertices are the points of $X$ and there is an oriented edge from $x$ to $y$ if and only if $x<y$ and there is no $z$ satisfying $x<z<y$. In subsequent Hasse diagrams we assume an upward orientation for simplicity.

A continuous map $f:X\rightarrow Y$ is a weak homotopy equivalence if the induced maps $f_*:\pi_i(X,x)\rightarrow \pi_i(Y,f(x))$ are isomorphisms for all $x\in X$ and all $i\leq 0$. Two topological spaces $X$ and $Y$ are weak homotopy equivalent if there is a sequence of weak homotopy equivalences such that $X \rightarrow X_1 \leftarrow X_2\cdots X_n\rightarrow Y$. This notion plays a central role in the theory of finite topological spaces. 

There is a natural functor $\mathcal{K}:F\rightarrow SimpComp$, where $SimpComp$ denotes the category whose objects are finite simplicial complexes and whose morphisms are simplicial maps. Given a finite topological space $X$, $\mathcal{K}(X)$ is the order complex of $X$, that is, the simplicial complex whose simplices are the non-empty chains of $X$. Let $|\mathcal{K}(X)|$ denotes the geometric realization of $\mathcal{K}(X)$. For every $u\in |\mathcal{K}(X)|$ there is a unique open simplex that contains $u$, let us denote this simplex by $(x_0,x_1,...,x_r)$, where $x_0<x_1<...<x_r$ in $X$. We let $f_X(u)=x_0$. In \cite{mccord1966singular} it is proved the following theorem.
\begin{thm}\label{thm:McCord} For every finite topological space $X$ there exists a weak homotopy equivalence $f_X:|\mathcal{K}(X)|\rightarrow X$.
\end{thm}
In the proof of Theorem \ref{thm:McCord} it is used the notion of basis-like open cover. An open cover $\mathcal{U}$ of $X$ will be called basis-like if whenever $x\in U\cap V$ and $U,V\in \mathcal{U}$, then there exists $W\in \mathcal{U}$ such that $x\in W\subseteq U\cap V$. For every finite topological space $X$ we get that $\mathbb{B}(X)=\{U_x |x\in X\}$ is an open cover basis-like. 

There is also a natural functor $\mathcal{X}:SimpComp\rightarrow F$ and a theorem similar to Theorem \ref{thm:McCord} (see \cite[Theorem 3]{mccord1966singular}). Given a simplicial complex $L$, $\mathcal{X}(L)$ is the face poset of $L$. If $L$ is a simplicial complex, then $\mathcal{K}(\mathcal{X}(L))$ is the barycentric subdivision of $L$. The finite barycentric subdivision of a finite topological space $X$ is defined by $\mathcal{X}(\mathcal{K}(X))$.

Given a topological space $X$, one useful construction to get an inverse system of polyhedra is the nerve of an open cover. Let $\mathcal{U}$ be an open cover of $X$, $N(\mathcal{U})$ denotes the nerve of $\mathcal{U}$. Its vertices are the elements $U$ of $\mathcal{U}$, and vertices $U_0,U_1,...,U_n\in \mathcal{U}$ span a simplex of $N(\mathcal{U})$ whenever $U_0\cap U_1 \cap... \cap U_n\neq \emptyset$. The set of open covers of $X$ is a directed set with the  relation of being a finer cover. Given two open covers $\mathcal{U}$ and $\mathcal{V}$ of $X$, we write $\mathcal{U}\leq \mathcal{V}$ if and only if $\mathcal{V}$ refines $\mathcal{U}$. If $\mathcal{U}\leq \mathcal{V}$, then one associates a simplicial map $p_{\mathcal{U}\mathcal{V}}:N(\mathcal{V})\rightarrow N(\mathcal{U})$ given by $p_{\mathcal{U}\mathcal{V}}(V)=U$, where $V\subseteq U$. This simplicial map is called projection. It is not unique, but every projection belongs to the same homotopy class \cite[Appendix 1, Section 3, Theorem 5]{mardevsic1982shape}. It is worth noting that the nerve of an open cover can be used to reconstruct the homotopy type of the topological space, the so-called nerve theorem (see \cite{borsuk1948imbedding} or \cite{mccord1967homotopy} for a more detailed discussion about open covers and simplicial complexes). Namely, if $\mathcal{U}$ is an open cover of a compact space $X$ such that every non-empty intersection of finitely many sets in $\mathcal{U}$ is contractible then $X$ is homotopy equivalent to $N(\mathcal{U})$.

Given a compact Hausdorff space $X$ one can associate an inverse system $C(X)=$ $(N(\mathcal{U})$, $[p_{\mathcal{U},\mathcal{U}'}], \Lambda)$ in $HPol$, called the \v{C}ech system of $X$, where $\Lambda$ denotes the set of all (normal) open covers ordered by the relation of being a finer cover and $p_{\mathcal{U},\mathcal{U}'}$ is a projection. This inverse system is relevant to define the shape category. As mentioned before, the morphisms of the shape category are given in terms of these inverse systems.

If $X$ is a compact topological space and $\mathcal{U}$ is a finite open cover of $X$, then $\mathcal{U}$ can be seen as a finite poset with the subset relation. From now on, we may treat finite open covers as finite topological spaces without explicit mention. We can associate to $\mathcal{U}$ the nerve of $\mathcal{U}$ or apply the functor $\mathcal{K}$ to get a simplicial complex. Nevertheless, if $X$ is a finite topological space, then the nerve complex does not reconstruct its homotopy type or its weak homotopy type. 
\begin{ex}\label{ex:preliminares} We consider the minimal finite model of the circle, that is, $X=\{a,b,c,d \}$ with its topology given by $\{ \emptyset, X, \{a \}, \{ b\}, \{c,a,b\}, \{d,a,b\},\{a,b\}  \}$. It is well-known that $X$ has the weak homotopy type of a circle and it is not contractible. We have that $U_a=\{ a\}$, $U_b=\{ b\},U_c=\{ c,a,b\}$ and $U_d=\{ d,a,b\}$, and $a,b<c,d$. It is easily seen that every open cover $\mathcal{V}$ of $X$ contains $U_c$ and $U_d$, or $U_c\cup U_d=X$. For the second case we get that $N(\mathcal{V})$ is contractible since $\mathcal{V}$ has a maximum, which is $U_c\cup U_d=X$. We leave it to the reader to verify that  $N(\mathcal{V})$ is contractible for every open cover $\mathcal{V}$ of $X$. Now, we consider the open cover $\mathbb{B}(X)=\{U_a,U_b,U_c,U_d \}$ as a poset with the subset relation. This poset is isomorphic to $X$. It follows easily that $\mathcal{K}(\mathbb{B}(X))$ is a triangulation of the circle and therefore it has the same weak homotopy of $X$.
\end{ex}

If we have two open covers $\mathcal{U}\leq \mathcal{V}$ of a topological space $X$, then we do not have necessarily a well-defined projection using the functor $\mathcal{K}$, that is, there is no continuous map $p_{\mathcal{U},\mathcal{V}}:\mathcal{K}(\mathcal{V})\rightarrow \mathcal{K}(\mathcal{U})$. 
\begin{ex} Let us consider the unit interval $I$ and the open sets $U_1=[0,\frac{2}{3})$ and $U_2=(\frac{1}{3},1]$. We define open covers of $I$ as follows: $\mathcal{U}=\{U_1,U_2\}$ and $\mathcal{V}=\{U_1,U_2,U_1\cap U_2 \}$. It is obvious that $\mathcal{V}$ refines $\mathcal{U}$. We get easily that $\mathcal{K}(\mathcal{U})$ is a  topological space with two points and $\mathcal{K}(\mathcal{V})$ is a triangulation of the unit interval. It is clear that there is no continuous map $p_{\mathcal{U},\mathcal{V}}:\mathcal{K}(\mathcal{V})\rightarrow \mathcal{K}(\mathcal{U})$ such that $p_{\mathcal{U},\mathcal{V}}(U)=V$ where $U\subseteq V$. Moreover, $N(\mathcal{U})$ and $N(\mathcal{V})$ are two triangulations of the unit interval.
\end{ex}

To conclude this section we show that every connected finite topological space has trivial shape.

\begin{prop}\label{prop:trivialShapeFinito} If $X$ is a connected finite topological space, then $X$ has trivial shape. 
\end{prop}
\begin{proof}
The result is an immediate consequence of the following: every continuous map from a connected finite topological space to a non-trivial compact metric space is the constant map. We prove the last assertion. Suppose $M$ is a compact metric space. If $f:X\rightarrow M$ is continuous at $x\in X$, then for every open neighborhood $V$ of $f(x)$ there exists an open neighborhood  $U$ of $x$ satisfying $f(U)\subseteq V$. We consider a decreasing sequence of positive values $\{\epsilon_n\}_{n\in \mathbb{N}}$ satisfying that $\lim_{n\rightarrow \infty}\epsilon_n=0$. Since $U_x$ is a minimal open neighborhood of $x$ in $X$, it follows that $f(U_x)\subseteq \mathcal{B}(f(x),\epsilon_n)$ for every $n\in \mathbb{N}$. Here $\mathcal{B}(f(x),\epsilon_n)$ denotes the open ball of radius $\epsilon_n$ and center $f(x)$. We have $\bigcap_{n\in \mathbb{N}}\mathcal{B}(x,\epsilon_n)=\{ f(x)\}$ and then $f(U_x)=f(x)$. Therefore, $f$ is a locally constant map defined on a connected space, which implies that it is constant. From this, we get that the only $HPol$-expansion of $X$ is the trivial one (for a deep discussion about the notion of $HPol$-expansion see \cite[Chapter 1, Section 2]{mardevsic1982shape}). Thus, $X$ has trivial shape.
\end{proof}
\begin{rem} From Proposition \ref{prop:trivialShapeFinito}, it follows easily that the shape category classifies a finite topological space by its number of connected components.
\end{rem}

\section{Category SW}\label{sec:construccionCategoria}

Firstly, we construct an auxiliary category $M$. The objects of $M$ are finite topological spaces. Given two finite topological spaces $X$ and $Y$ we define $M(X,Y)=\{f:|\mathcal{K}(X)|\rightarrow |\mathcal{K}(Y)|| f$ is a continuous map$\}$. It is easy to check that $M$ is a category. Let $HM$ denote the homotopy category of $M$, that is, the objects of $HM$ are the objects of $M$ and the morphisms of $HM$ are the homotopy classes of the morphisms of $M$. Again, it is trivial to check that $HM$ is a category.
\begin{rem} It is easily seen that $HM$ is a category that classifies finite topological spaces by their weak homotopy type.
\end{rem}

Given a compact space $X$ we will associate to $X$ an object of pro-$HM$. We consider $Cov(X)=\{ \mathcal{U}$ is a finite open cover of $X|$ $\mathcal{U}$ is basis-like$\}$. We write $\mathcal{U}\geq_C\mathcal{V}$ if and only if $\mathcal{U},\mathcal{V}\in Cov(X)$, $\mathcal{U}$ refines $\mathcal{V}$ and there exists a continuous map $p_{\mathcal{V},\mathcal{U}}:\mathcal{U}\rightarrow \mathcal{V}$ such that for every $U\in \mathcal{U}$ we get that $p_{\mathcal{V},\mathcal{U}}(U)\in \mathcal{V}$  contains $U$. We will omit the subscript of $p_{\mathcal{V},\mathcal{U}}$ when no confusion can arise.

\begin{ex} Let us consider $X=[0,1]\times [0,1]\subset \mathbb{R}^2$. Let $\mathcal{U}$ denote the open cover given by $U_1=[0,1]\times [0,\frac{3}{4})$ and $U_2=[0,1]\times (\frac{1}{4},1]$. Let $\mathcal{V}$ denote the open cover given by $U_1$, $U_2$ and $U_3=U_1\cap U_2$. It follows immediately that $\mathcal{V}\ngeq_C \mathcal{U}$ because there is no continuous map $p:\mathcal{V}\rightarrow \mathcal{U}$ satisfying that $U\subseteq p(U)$. In Figure \ref{fig:ejemploCoverNoRefina} we present a schematic representation of the situation described above.
\begin{figure}[h]
\centering 
\includegraphics[scale=1]{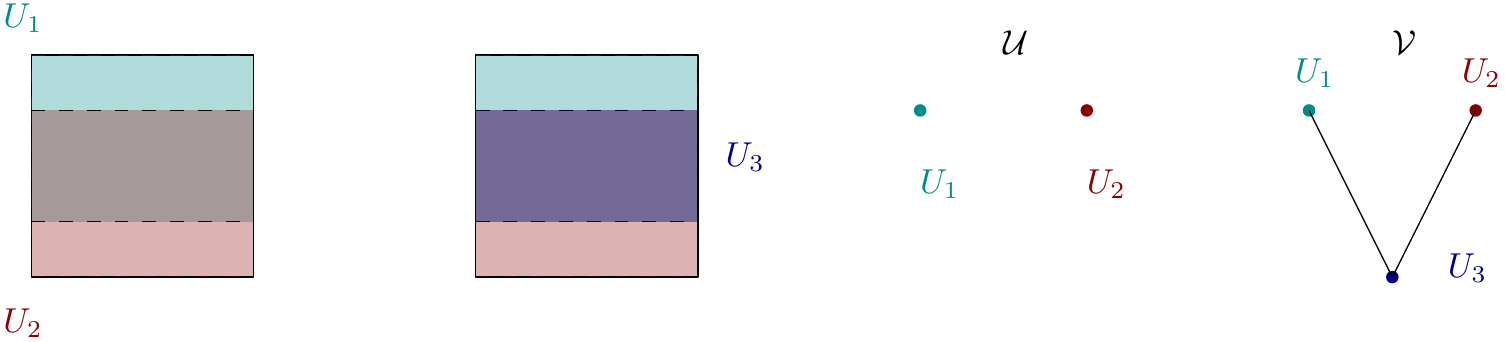}\caption{$X$, $\mathcal{U}$, $\mathcal{V}$ and Hasse diagrams of $\mathcal{U}$ and $\mathcal{V}$.}\label{fig:ejemploCoverNoRefina}
\end{figure}
\end{ex}

\begin{rem}\label{rem:coverhomeo} Let $X$ be a finite topological space and $\mathbb{B}(X)=\{U_x|x\in X \}$. Then $\mathbb{B}(X)\in Cov(X)$. In addition, a trivial verification shows that $\mathbb{B}(X)$ is homeomorphic to $X$. We have that $\varphi:\mathbb{B}(X)\rightarrow X$ given by $\varphi(U)=max(U)$ is a homeomorphism, where $max(\cdot)$ denotes the maximum of a poset.
\end{rem}

We prove some properties of the relation $\leq_C$ defined on $Cov(X)$ when $X$ is a finite topological space.

\begin{prop}\label{prop:cofinal} Let $X$ be a finite topological space. Then $\mathbb{B}(X)\geq_C \mathcal{U}$ for every $\mathcal{U}\in Cov(X)$. 
\end{prop}
\begin{proof}
Let $\mathcal{U}\in Cov(X)$. It is trivial to check that $\mathbb{B}(X)$ refines $\mathcal{U}$, in fact, $\mathbb{B}(X)$ is a basis for the topology of $X$. We define $p:\mathbb{B}(X)\rightarrow \mathcal{U}$ as follows: $p(U_x)$ is the minimum open set in $\mathcal{U}$ containing $U_x$. We prove  that $p$ is well-defined, i.e., $p(U_x)$ exists. Since $\mathbb{B}(X)$ refines $\mathcal{U}$, it follows that there exists at least one open set in $\mathcal{U}$ that contains $U_x$. If $U_x\in \mathcal{U}$, then $p(U_x)=U_x$ and $p$ is well-defined. If $U_x\notin \mathcal{U}$, then we argue by contradiction. Suppose that there is no minimum open set containing $U_x$, that is, there are two open sets $U_1, U_2\in \mathcal{U}$ that are not comparable satisfying $U_x\subseteq U_i$  for $i=1,2$ and the property that there is no $W\in \mathcal{U}$ such that $U_x\subseteq W\subseteq U_i$  for $i=1,2$. We get that $U_1\cap U_2$ is non-empty because it contains $U_x$. By the property of being basis-like, we have $U_1\cap U_2=\bigcup_{t\in T}U_t$ for some $T$, where $U_t\in \mathcal{U}$ for every $t\in T$. On the other hand, $\mathbb{B}(X)$ is a basis for the topology in $X$, so $U_t=\bigcup_{y\in I_t}U_y$  for every $t\in T$ and some $I_t\subseteq X$. We get that  $U_x\subseteq \bigcup_{t\in T}\bigcup_{y\in I_t} U_y$, which means that $x\in U_y$ for some $y \in I_t$ and $t\in T$. Hence, $U_x\subseteq U_y\subseteq U_t\in \mathcal{U}$, but $U_t\subseteq U_i$ for $i=1,2$, which leads to contradiction. 

We have shown that $p$ is well-defined, it remains to show the continuity of it. It suffices to check that it is order-preserving. Again, we argue by contradiction. Suppose that $U_x\subseteq U_y$ and $p(U_x)\nsubseteq p(U_y)$. By hypothesis, $U_x\subseteq U_y\subseteq p(U_y)$ so $p(U_x)\cap p(U_y)\neq \emptyset$. We repeat the same argument used before to get the contradiction. We have that $p(U_x)\cap p(U_y)=\bigcup_{t\in T} U_t$ for some $T$, where $U_t\in \mathcal{U}$ for every $t\in T$. On the other hand, $U_t=\bigcup_{z\in I_t} U_z$ for some $I_t\subseteq X$, where $U_z\in \mathbb{B}(X)$ for every $z\in I_t$ and every $t\in T$. We get $U_x\subseteq p(U_x)\cap p(U_y)=\bigcup_{t\in T}\bigcup_{z\in I_t}U_z$, therefore,  $U_x\subseteq U_z\subset U_t \subseteq p(U_x)$ for some $t\in T$, which leads to a contradiction with the minimality of $p(U_x)$.
\end{proof}

\begin{prop}\label{prop:cofinal2} Let $X$ be a finite topological space. If there exists $\mathcal{U}\in Cov(X)$ with $\mathcal{U}\geq_C \mathbb{B}(X)$, then $\mathcal{U}$ is homotopy equivalent to $\mathbb{B}(X)$.  
\end{prop}
\begin{proof}
We use Proposition \ref{prop:characterizationHomotopy}. Then it suffices to show that $p_{\mathbb{B},\mathcal{U}}\circ p_{\mathcal{U},\mathbb{B}}\geq id_{\mathcal{U}}$ and $p_{\mathcal{U},\mathbb{B}}\circ p_{\mathbb{B},\mathcal{U}}\geq id_{\mathbb{B}}$. By construction, for every $U\in \mathcal{U}$ we get  that $id(U)=U\subseteq p_{\mathbb{B},\mathcal{U}}( p_{\mathcal{U},\mathbb{B}}(U))$ because $U\subseteq p_{\mathcal{U},\mathbb{B}}(U)\subseteq  p_{\mathbb{B},\mathcal{U}}( p_{\mathcal{U},\mathbb{B}}(U))$. We apply the same argument to show $p_{\mathcal{U},\mathbb{B}}\circ p_{\mathbb{B},\mathcal{U}}\geq id_{\mathbb{B}}$. Thus, $\mathbb{B}(X)$ and $\mathcal{U}$ have the same homotopy type.
\end{proof}

\begin{prop}\label{prop:directedSetFinite} Let $X$ be a finite topological space. Then $Cov(X)$ with the relation $\leq_C$ is a directed set.
\end{prop}
\begin{proof}
The reflexive property follows trivially.  For the transitive property we only need to compose the continuous maps given by the relations $\mathcal{U}\geq_C\mathcal{V}$ and $\mathcal{V}\geq_C \mathcal{W}$, where $\mathcal{U}, \mathcal{V}, \mathcal{W}\in Cov(X)$. By Proposition \ref{prop:cofinal} it is obvious that for any $\mathcal{U},\mathcal{V}\in Cov(X)$ there exists $\mathcal{W}\in Cov(X)$ such that $\mathcal{W}\geq_C\mathcal{U}$ and  $\mathcal{W}\geq_C\mathcal{V}$.
\end{proof}

If $X$ is a non-finite compact topological space, then we define $\overline{Cov}(X)=\{ \mathcal{U}$ is a finite open cover of $X|$ for every $U,V\in \mathcal{U}$, $U\cap V\in \mathcal{U}\}$. We have trivially that $\overline{Cov}(X)\subseteq Cov(X)$. In fact, $\overline{Cov}(X)$ is cofinal in $Cov(X)$ with the usual relation of refinement. 

\begin{lem}\label{lem:direcetd}  If $X$ is a compact Hausdorff space, then $(\overline{Cov}(X),\geq_C)$ is a directed set.
\end{lem}
\begin{proof}

We prove that $(\overline{Cov}(X),\geq_C)$ is a directed set. For any $\mathcal{U}$, $\mathcal{V}\in \overline{Cov}(X)$ there exists a finite open cover $\mathcal{W}$ that refines $\mathcal{U}$ and $\mathcal{V}$ in the usual sense \cite [Appendix 1, Section 3]{mardevsic1982shape}. Let $\overline{\mathcal{W}}$ denote the open cover $\mathcal{W}$ with all the possible intersections of open sets from it. Then $\overline{\mathcal{W}}$ clearly refines $\mathcal{W}$ and $\overline{\mathcal{W}}\in \overline{Cov}(X)$. We define $p:\overline{\mathcal{W}}\rightarrow \mathcal{U}$ given by 
$$p(W)=\bigcap_{V\in U_W} V, \ \text{where}  \ U_W=\{V\in \mathcal{U}|W\subseteq V\}. $$
The map $p$ is well-defined because $\mathcal{U}\in \overline{Cov}(X)$, which means that the intersection of elements of $\mathcal{U}$ is again an element of $\mathcal{U}$. Therefore, $p(W)\in \mathcal{U}$ for every $W\in \overline{\mathcal{W}}$. We prove the continuity of $p$. By hypothesis $W'\subseteq W$ implies $U_{W}\subseteq U_{W'}$ and therefore $p(W')\subseteq p(W)$. Thus, $\overline{\mathcal{W}}\geq_C \mathcal{U}$ and we only need to repeat the same argument to get $\overline{\mathcal{W}}\geq_C \mathcal{V}$. 

It is clear that $\leq_C$ satisfies the reflexive property. The proof to show that $\leq_C$ also satisfies the transitivity property is straightforward. 
\end{proof}

With the usual notion of refinement we can have the following situation. Given a finite topological space $X$ there is a finite cover $\mathcal{U}\in Cov(X)$ satisfying that $\mathcal{U}$ is cofinal in $Cov(X)$ and $\mathcal{U}\neq \mathbb{B}(X)$. On the other hand, $\mathbb{B}(X)$ is also cofinal in $Cov(X)$. We have that $\mathcal{U}$ and $\mathbb{B}(X)$ do not have necessarily the same homotopy type or the same weak homotopy type. Therefore, if we consider a cofinal element in $Cov(X)$, then we make a choice that it is not unique.
In the following example we illustrate this situation.

\begin{ex}\label{example:s1covers} We consider the finite topological space introduced in Example \ref{ex:preliminares} and $\mathbb{B}(X)\in Cov(X)$. We get that $\mathbb{B}(X)$ is homeomorphic to $X$, see Remark \ref{rem:coverhomeo}. We consider $\mathcal{U}=\{ \{ a,b\}, \{a,b,c \},\{a,b,d \} \}\in Cov(X)$. Since $\mathcal{U}$ contains a minimum, which is $\{ a,b\}$, it follows that $\mathcal{U}$ has the same homotopy type of a point. We have that $|\mathcal{K}(X)|=|\mathcal{K}(\mathbb{B}(X))|=S^{1}$, which implies that $\mathbb{B}(X)$ is not contractible. It is not difficult to get that $\mathbb{B}(X)$ refines $\mathcal{U}$ and $\mathcal{U}$ refines $\mathbb{B}(X)$ in the usual sense. Moreover, $\mathbb{B}(X)\geq_C\mathcal{U}$ considering $p:\mathbb{B}(X)\rightarrow \mathcal{U}$ given by $p(\{ a,b,c\})=\{ a,b,c\}$, $p(\{ a,b,d\})=\{ a,b,d\}$ and $p(a)=p(b)=\{a,b\}$ (or see Proposition \ref{prop:cofinal}). There is no continuous map $q:\mathcal{U}\rightarrow \mathbb{B}(X)$ satisfying that $U\subseteq q(U)$ for every $U\in \mathbb{B}(X)$. If it exists, then $q(\{ a,b,c\})=\{ a,b,c\}$ and $q(\{ a,b,d\})=\{ a,b,d\}$. We only have two options for $q(\{ a,b\})$. If $q(\{ a,b\})=q(\{ a,b,c\})$, then $q$ is not continuous since $\{ a,b,d\}>\{ a,b\}$ but $q(\{ a,b,d\})=\{ a,b,d\}\ngeq \{a,b,c \}=q(\{ a,b\})$. Similarly, we get a contradiction with the continuity of $q$ for the other case. In Figure \ref{fig:esquemaNorefina} we have the Hasse diagrams of $X$, $\mathbb{B}(X)$ and $\mathcal{U}$.
\begin{figure}[h]
\centering
\includegraphics[scale=1]{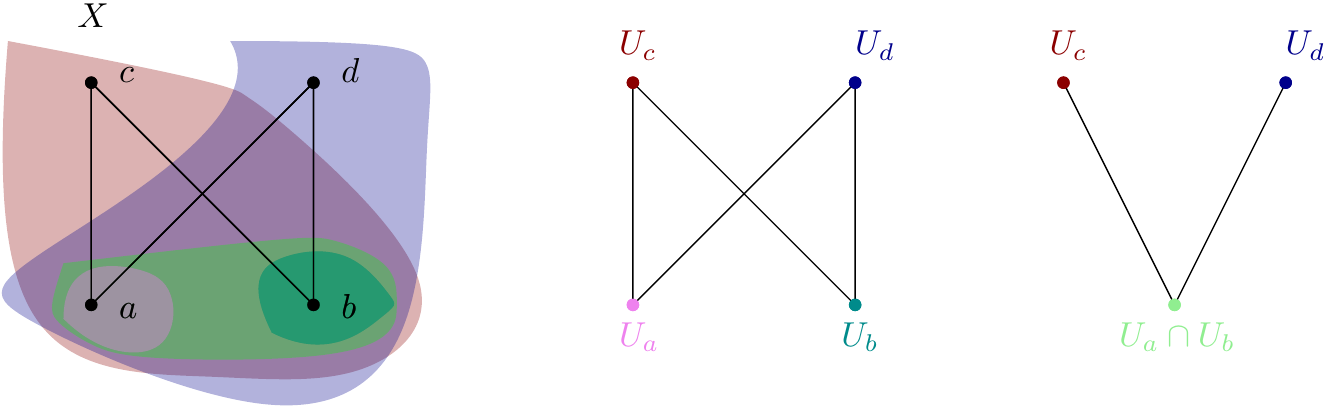}
\caption{Hasse diagrams of $X$, $\mathbb{B}(X)$ and $\mathcal{U}$, and schematic representation of the open sets of $\mathbb{B}(X)$ and $\mathcal{U}$ in the Hasse diagram of $X$.}\label{fig:esquemaNorefina}

\end{figure}
\end{ex}


Example \ref{example:s1covers} shows that for the finite case it does not make sense to consider a cofinal element in $Cov(X)$ with the usual notion of refinement. This is the main reason why we define the relation $\leq_C$. Moreover, for compact Hausdorff spaces the relation $\leq_C$ fits well with $\overline{Cov}(X)$, which is a cofinal subset in $Cov(X)$ with the usual notion of refinement. In this case, it makes sense to consider a cofinal subset because we are not making a choice of an element that determines $Cov(X)$. This is due to the cardinality of $Cov(X)$. For the finite case $Cov(X)$ is a finite set, i.e., we do not have enough covers, while for the non-finite case the cardinality of $Cov(X)$ is infinite. Therefore, for non-finite topological spaces in the remainder of this section we assume that $Cov(X)=\overline{Cov}(X)$.  
\begin{rem}\label{rem:BondingBienDefinidas} Let $X$ be a compact Hausdorff space and  $\mathcal{U},\mathcal{V}\in Cov(X)$. If $\mathcal{U}$ refines $\mathcal{V}$, then $\mathcal{U}\geq_C \mathcal{V}$ considering $p:\mathcal{U}\rightarrow \mathcal{V}$ given by $p(U)=\bigcap_{U\subseteq V\in\mathcal{V}} V$. By construction, $p$ is well-defined and continuous. Furthermore, if $q:\mathcal{U}\rightarrow\mathcal{V}$ is another map given by the relation $\geq_C$, then it is clear that $p\leq q$, which implies that $p$ is homotopic to $q$.  This result can be seen as an analog of the uniqueness (in homotopy) of the projections in the construction of the \v{C}ech system of a compact Hausdorff space (see \cite[Appendix 1, Section 3, Theorem 5]{mardevsic1982shape}).
\end{rem}
For every compact space $X$ we associate an element in pro-$HM$ given by 

$$\mathcal{M}(X)= ( \mathcal{U},[|\mathcal{K} (p_{\mathcal{U},\mathcal{V}})|],Cov(X)), $$
where $\mathcal{K}(p_{\mathcal{U},\mathcal{V}}):\mathcal{K}(\mathcal{U})\rightarrow\mathcal{K}( \mathcal{V})$ denotes the induced continuous function given the functor $\mathcal{K}$ and $p_{\mathcal{U},\mathcal{V}}$ is a continuous map from $\mathcal{V}$ to $\mathcal{U}$ given by the relation $\leq_C$. It remains to verify that $\mathcal{M}(X)$ is indeed an element of pro-$HM$.

\begin{lem}\label{lem:elementof}
If $X$ is a compact Hausdorff space or a finite topological space, then $\mathcal{M}(X)$ is an object of pro-$HM$.

\end{lem}
\begin{proof}
In Lemma \ref{lem:direcetd}, it is shown that $Cov(X)$ with the relation $\leq_C$ is a directed set. It only remains to show that for every $\mathcal{U},\mathcal{V},\mathcal{W} \in Cov(X)$ satisfying that $\mathcal{U}\leq_C \mathcal{V}$ and $\mathcal{V}\leq_C\mathcal{W}$ we get $\mathcal{K} (p_{\mathcal{U},\mathcal{V}})\circ \mathcal{K}( p_{\mathcal{V},\mathcal{W}}) \simeq \mathcal{K}(p_{\mathcal{U},\mathcal{W}})$.

We study the non-finite case. By the definition of $\leq_C$ and Remark \ref{rem:BondingBienDefinidas} it is obvious that $p_{\mathcal{U},\mathcal{W}}(W) \subseteq p_{\mathcal{U},\mathcal{V}}(p_{\mathcal{V},\mathcal{W}} (W))$ for every $W\in \mathcal{W}$. Therefore, $p_{\mathcal{U},\mathcal{W}}\simeq p_{\mathcal{U},\mathcal{V}} \circ p_{\mathcal{V},\mathcal{W}}$ by Proposition \ref{prop:characterizationHomotopy}, which implies the desired result. 

We prove the result for the finite case. By Proposition \ref{prop:cofinal2} it follows that $\mathbb{B}(X)$ is cofinal in $Cov(X)$. Thus, $\mathcal{M}(X)$ is a rudimentary system.
\end{proof}
Applying the functor $\mathcal{K}$, $\mathcal{M}(X)$ can also be seen as an element of $pro$-$HPol$. Let $\mathcal{K}(X)$ denote the inverse system given by
$$ \mathcal{K}(X)=(  |\mathcal{K}(\mathcal{U})|,[|\mathcal{K} (p_{\mathcal{U},\mathcal{V}})|],Cov(X) ). $$
\begin{prop}\label{prop:hpol} If $X$ is a compact Hausdorff space, then $\mathcal{K}(X)$ is a $HPol$-expansion.

\end{prop}
\begin{proof}
We prove the result showing that $\mathcal{K}(X)$ is isomorphic to the  \v{C}ech system $( |N(\mathcal{U})|,$ $N(p_{\mathcal{U},\mathcal{V}}), \overline{Cov}(X) )$, which is a $HPol$-expansion of $X$. A detailed description about \v{C}ech systems for compact Hausdorff spaces can be found in \cite[Appendix 1.3]{mardevsic1982shape}.

Firstly, we study the map $N(p_{\mathcal{U},\mathcal{V}})$. We have that $p_{\mathcal{U},\mathcal{V}}:\mathcal{V}\rightarrow \mathcal{U}$ satisfies $V\subseteq p(V)$ for every $V\in \mathcal{V}$. Then, $p_{\mathcal{U},\mathcal{V}}$ induces a map between the nerves, i.e., $N(p):|N(\mathcal{V})|\rightarrow |N(\mathcal{U})|$. By construction, $\mathcal{K}(\mathcal{U})\subseteq N(\mathcal{U})$. Moreover, $|\mathcal{K}(\mathcal{U})|$ is a deformation retract of $|N(\mathcal{U})|$ by  \cite[Lemma 3]{mccord1967homotopy}. Furthermore, the set of vertices of  $\mathcal{K}(\mathcal{U})$ and $ N(\mathcal{U})$ are the same. We have that $p$ can be taken satisfying that $N(p)$ and $\mathcal{K}(p)$ are the same map over the set of vertices. In addition, in \cite[Appendix 1.3]{mardevsic1982shape}, it is proved that all choices of $N(p)$ are homotopic. Thus, we have $\mathcal{K}(p)= N(p)_{|\mathcal{K}(\mathcal{U})}$. 

We also have a level morphism between the two inverse systems given by the natural inclusion $i$. We will use Morita's lemma (see \cite{morita1974hurewicz} or \cite[Chapter 2, Theorem 5]{mardevsic1982shape}) to conclude the proof.

\[
  \begin{tikzcd}[row sep=huge,column sep=huge]
\mathcal{K}(\mathcal{U}) \arrow[d,"i"] & \mathcal{K}(\mathcal{V}) \arrow[l,"\mathcal{K}(p_{\mathcal{U},\mathcal{V}})"] \arrow[d,"i"'] \\
N(\mathcal{U}) & N(\mathcal{V}) \arrow[l,"N(p_{\mathcal{U},\mathcal{V}})"] \arrow[ul, "g_\mathcal{U}"]
  \end{tikzcd}
\]
We consider $g_\mathcal{U}=r \circ N(p)$, where $r:|N(\mathcal{U})|\rightarrow |\mathcal{K}(\mathcal{U})|$ denotes the retraction considered before. The commutative up to homotopy of the above diagram follows trivially.
\end{proof}

We define the category $SW$, its objects are compact spaces and if $X,Y\in Obj(\mathcal{S}\mathcal{W})$, then $SW(X,Y)=\{ f:\mathcal{M}(X) \rightarrow \mathcal{M}(Y)|$ $f$ is a morphism in pro-$HM \}$. 

\begin{thm}\label{thm:classifyFinite}
Let $X$ and $Y$ be finite topological spaces. Then $X$ is weak homotopy equivalent to $Y$ if and only if $X$ is isomorphic to $Y$ in $SW$. 
\end{thm}
\begin{proof} Suppose $X$ and $Y$ are finite topological spaces that are isomorphic in $SW$. We prove that $X$ and $Y$ are weak homotopy equivalent, i.e., there exists a $CW$-complex $Z$ and weak homotopy equivalences $Z\rightarrow X$, $Z\rightarrow Y$ (\cite[Proposition 4.13 and Corollary 4.19 ]{hatcher2000algebraic}). 
By hypothesis, $|\mathcal{K}(\mathbb{B}(X))|$ and $|\mathcal{K}(\mathbb{B}(Y))|$ are homotopy equivalent. In addition, by Theorem \ref{thm:McCord} there exist weak homotopy equivalences $|\mathcal{K}(\mathbb{B}(X))|\rightarrow X$, $|\mathcal{K}(\mathbb{B}(Y))|\rightarrow Y$. From here, we deduce the desired result. Now, we prove the opposite, so let us assume that $X$ is weak homotopy equivalent to $Y$. Hence, $|\mathcal{K}(\mathbb{B}(X))|=|\mathcal{K}(X)|$ and $|\mathcal{K}(\mathbb{B}(Y))|=|\mathcal{K}(Y)|$ are also weak homotopy equivalent. By \cite{hatcher2000algebraic}, there exists a CW-complex $Z$ and weak homotopy equivalences $|\mathcal{K}(X)| \leftarrow Z\rightarrow |\mathcal{K}(Y)|$. By a well-known theorem of Whitehead a weak homotopy equivalence between connected CW-complexes is a homotopy equivalence, therefore, we have that the previous weak homotopy equivalences are indeed homotopy equivalences. Then, $|\mathcal{K}(\mathbb{B}(X))|$ is homotopy equivalent to $|\mathcal{K}(\mathbb{B}(Y))|$, so $X$ is isomorphic to $Y$ in $SW$.
\end{proof}
\begin{rem}
It is a simple matter to check that in $SW$ polyhedra are isomorphic to finite topological spaces.
\end{rem}

\begin{thm} 
Let $X$ and $Y$ be compact Hausdorff spaces. Then $X$ and $Y$ have the same shape if and only if $X$ and $Y$ are isomorphic in $SW$.
\end{thm}
\begin{proof}
Suppose $X$ and $Y$ are compact Hausdorff spaces that are isomorphic in $SW$. Then $\mathcal{K}(X)$ and $\mathcal{K}(Y)$ are $HPol$-expansions of $X$ and $Y$ by Proposition \ref{prop:hpol}. It is clear that $\mathcal{M}(X)\simeq \mathcal{M}(Y)$ implies $\mathcal{K}(X)\simeq \mathcal{K}(Y)$. Thus, $X$ and $Y$ have the same shape. Suppose $X$ and $Y$ are compact Hausdorff spaces that have the same shape. Then, there are two $HPol$-expansions $(X_\lambda, p_{\lambda\lambda'},\Lambda )$ and $(Y_\lambda, p_{\lambda\lambda'},\Lambda)$ satisfying that $(X_\lambda, p_{\lambda\lambda'},\Lambda )$ is isomorphic to $(Y_\lambda, p_{\lambda\lambda'},\Lambda)$. For each polyhedron $X_\lambda$ of $(X_\lambda, p_{\lambda\lambda'},\Lambda)$ we take a triangulation and we apply the functor $\mathcal{X}$ to get a finite topological space. It is clear that $(\mathcal{X}(X_\lambda), p_{\lambda\lambda'},\Lambda))$ is an element of pro-$HM$. In addition, $(|\mathcal{K}(\mathcal{X}(X_\lambda))|,p_{\lambda\lambda'},\Lambda)$ is isomorphic to  $(X_\lambda, p_{\lambda\lambda'},\Lambda)$ in pro-$HPol$ because $\mathcal{K}(\mathcal{X}(X_\lambda))$ is just the barycentric subdivision of the triangulation given to $X_\lambda$, i.e., $|\mathcal{K}(\mathcal{X}(X_\lambda))|=X_\lambda$. In fact, $( \mathcal{X}(X_\lambda), p_{\lambda\lambda'},\Lambda )$ is isomorphic to $\mathcal{M}(X)$ because $(|\mathcal{K}(\mathcal{U})|, \mathcal{K}(p), \overline{Cov}(X) )$ is a $HPol$-expansion of $X$ by Proposition \ref{prop:hpol}. The same situation holds for $Y$. Thus, $\mathcal{M}(X)$ is isomorphic to $\mathcal{M}(Y)$ in pro-$HM$ and we deduce that $X$ is isomorphic to $Y$ in $SW$.
\end{proof}

Let $hm:F\rightarrow HM$ denote the functor given by $hm(X)=X$ and $hm(f)=|\mathcal{K}(f)|$, where $X$ is a finite topological space and $f:X\rightarrow Y$ is a continuous map. Let $sw$ denote the natural functor between $Sh$ and $SW$. This functor induces an isomorphism between the full subcategories of $Sh$ and $SW$ restricted to compact Hausdorff spaces. Let $i:HM\rightarrow SW$ denote the functor given by $i(X)=X$ and $i(f)=f$.   

We summarize the relations between the categories mentioned throughout this note in the following diagram.

\[
  \begin{tikzcd}[row sep=large,column sep=huge]
Poset \simeq F \arrow[dr,"hm"'] \arrow[r, yshift=0.7ex,"\mathcal{K}"]& SimpComp \arrow[l, yshift=-0.7ex,"\mathcal{X}"] \arrow[r,"|\cdot|"] & HPol  \arrow[r,"S"]&  Sh \arrow[r,"sw"] & SW \\
          &   HM   \arrow[rrru,"i"]  &      &     & 
  \end{tikzcd}
\]

\bibliography{bibliografia}
\bibliographystyle{plain}

\newcommand{\Addresses}{{
  \bigskip
  \footnotesize

  \textsc{ P.J. Chocano, Departamento de Matemática Aplicada,
Ciencia e Ingeniería de los Materiales y
Tecnología Electrónica, ESCET
Universidad Rey Juan Carlos, 28933
Móstoles (Madrid), Spain}\par\nopagebreak
  \textit{E-mail address}:\texttt{pedro.chocano@urjc.es}
  
   \medskip
   
\textsc{ M. A. Mor\'on,  Departamento de \'Algebra, Geometr\'ia y Topolog\'ia, Universidad Complutense de Madrid and Instituto de
Matematica Interdisciplinar, Plaza de Ciencias 3, 28040 Madrid, Spain}\par\nopagebreak
  \textit{E-mail address}: \texttt{ma\_moron@mat.ucm.es}

  \medskip

\textsc{ F. R. Ruiz del Portal,  Departamento de \'Algebra, Geometr\'ia y Topolog\'ia, Universidad Complutense de Madrid and Instituto de
Matematica Interdisciplinar
, Plaza de Ciencias 3, 28040 Madrid, Spain}\par\nopagebreak
  \textit{E-mail address}: \texttt{R\_Portal@mat.ucm.es}

}}

\Addresses

\end{document}